\documentclass[reqno, 12pt]{amsart}
\usepackage{amsmath,amssymb,amsxtra,amsfonts,amsthm,url}
\usepackage{tikz}
\usetikzlibrary{arrows,backgrounds,decorations,automata}
\usepackage{hyperref}
\usepackage{natbib}
\usepackage{enumerate}

\newtheorem{theorem}{Theorem}[section]
\newtheorem{lemma}[theorem]{Lemma}
\newtheorem{corollary}[theorem]{Corollary}
\newtheorem{proposition}[theorem]{Proposition}

\newtheorem{example}[theorem]{Example}

\theoremstyle{definition}
\newtheorem{remark}[theorem]{Remark}
\newtheorem{definition}[theorem]{Definition}

\newcommand{\BbC}{\mathbb{C}}

\newcommand{\NN}{\mathbb{N}}

\newcommand{\ZZ}{\mathbb{Z}}

\newcommand{\CC}{\mathbb{C}}

\newcommand{\grdr}{{\rm gr}}

\DeclareMathOperator{\lm}{LM}

\renewcommand{\d}{\partial}

\title[On Noncommutative Finite Factorization Domains]{On Noncommutative Finite Factorization Domains}

\author{Jason P.  Bell}
\address{Department of Pure Mathematics\\
        University of Waterloo\\
        Waterloo, ON, Canada}
       \email{jpbell@uwaterloo.ca}
        \author{Albert Heinle}
     \address{David R. Cheriton School of Computer Science\\
             University of Waterloo\\
        Waterloo, ON, Canada}
        \email{aheinle@uwaterloo.ca}
     
\author{Viktor Levandovskyy}
\address{Lehrstuhl D f\"ur  Mathematik \\
  RWTH Aachen University\\
  Aachen,  Germany}
  \email{viktor.levandovskyy@math.rwth-aachen.de}
\thanks{The first two authors thank NSERC for its generous support}

\date{\today}
\begin{document}
\begin{abstract} A domain $R$ is said to have the finite factorization
  property if every nonzero non-unit element of $R$ has at least one
  and at most finitely many distinct factorizations up to
  multiplication of irreducible factors by central units. 
 Let $k$ be an algebraically closed field and let $A$ be a
 $k$-algebra.  We show that if $A$ has an associated graded ring that
 is a domain with the property that the dimension of each homogeneous
 component is finite then $A$ is a finite factorization domain. 
 As a corollary, we show that many
 classes of algebras have the finite factorization property, including
 Weyl algebras, enveloping algebras of finite-dimensional Lie
 algebras, quantum affine spaces and shift algebras. 
 This provides a termination criterion for factorization procedures over these algebras.  
 In addition, we give explicit upper bounds on
 the number of distinct factorizations of an element in terms of data
 from the filtration. 
 \end{abstract}

\maketitle
\tableofcontents

\section{Introduction}
The study of counting the number of factorizations of elements in
rings has been a topic of great investigation
\citep{anderson1990factorization, anderson1992elasticity, anderson1996finite, tsarev1996algorithm, anderson1997factorization}.  The ideal behaviour in this setting
occurs in unique factorization domains, where up to multiplication by
central units in the rings and reordering of factors, there is exactly one
factorization of a nonzero, non-unit element into (finitely many) irreducible factors.
Consequently, if one studies distinct factorizations of elements into
factors then there are only finitely many such decompositions.  In
general, a (not necessarily commutative) domain with the property that
each nonzero element that is not a unit has at least one and only
finitely many distinct factorizations (see Definition \ref{def: FFD}
for a precise definition) is called a \emph{finite factorization
domain} (FFD, for short). In the commutative setting, factorization properties of
several types of integral domains have been studied in the past. A
very extensive collection of terminology has been developed in
\citep{anderson1997factorization} to
categorize these
domains. Furthermore, necessary and sufficient conditions for the property of
being a finite factorization domain have been discovered. For example,
every Krull domain is a finite factorization domain and a polynomial
ring over a finite factorization domain is again a finite
factorization domain \citep{anderson1990factorization}.
It is to be remarked that the community dealing with
factorization of commutative domains refined the span between unique factorization domains
and domains, where infinitely many factorizations of a single element
are possible, even further. We content ourselves in this paper
to identify $k$-algebras, where $k$ is a field, with the finite factorization
domain property. Up till now considerably less is known in the
noncommutative setting. P.~M.~Cohn proved that the free associative algebra over a field
is even a unique factorization domain \citep[Theorem 6.3]{cohn1963noncommutative}. 
In the recent paper by Baeth and Smertnig \citep{BS14}, containing an overview of recent results, 
a generalization of commutative theory as of Anderson et al. as well as its limits is provided. Typical examples considered
in that paper are matrix rings over certain domains. Perhaps the best-known result 
for a noncommutative algebra being a finite factorization domain (that is, on the other side of the
spectrum of algebras in certain sense) is the result of S.~Tsarev \citep{tsarev1996algorithm},
from which it can be derived, that the first Weyl algebra over $\BbC$
--- that is  an algebra with generators $x$ and $\partial$ and with
the relation $\partial x-x\partial=1$ --- is a finite factorization domain.  


Nonetheless, the proof that the $n$-th Weyl algebra with $n\geq 2$ is still a
finite factorization domain is more involved. In this paper, we will
show for a broad class of algebras (which ubiquitously appear in applications of e.~g. algebraic analysis), including the $n$-th
Weyl algebras, that their nonzero non-unit elements have only 
finitely many distinct factorizations up to multiplication by a
central unit.


To describe our generalization, we must briefly discuss the notions of a filtration of an algebra and of a graded algebra.  Let $k$ be a field and let $A$ be a $k$-algebra.   We say that $A$ has an ($\NN$-)\emph{filtration} if $A$ is a union of $k$-subspaces $V=\{V_n \colon n\in\NN\}$ such that
$V_n\subseteq V_m$ whenever $n\leq m$ and $V_nV_m\subseteq V_{n+m}$ for
all $m,n\ge 0$. If the stronger property $A =
\bigoplus_{n=0}^{\infty}V_n$ holds, i.e. $A$ is not just a union of
the $V_i$ but their direct sum for $i \in \NN$, we say that $A$ has an
$\NN$-\emph{grading} and refer to $A$ as \emph{graded algebra}.

%
To any filtration one can form an \emph{associated graded
  algebra}, denoted by ${\rm gr}_V(A)$, which is given by
$${\rm gr}_V(A) = \bigoplus_{n=0}^{\infty} V_{n}/V_{n-1},$$ where we take $V_{-1}=\{0\}$.  It is well-known that many of the properties of the associated graded algebra can be lifted to $A$.  For example, a classical result
shows that if the associated graded algebra of $A$ is a (not
necessarily commutative) domain then so is $A$ \citep[Proposition
  1.6.6]{mcconnell2001noncommutative}. We say that a filtration (resp. a grading) $\{V_n \colon n\in\NN\}$ is \emph{finite-dimensional} if each $V_{n}$ is a finite-dimensional $k$-vector space.
Our main result is the following.
\begin{theorem}
Let $k$ be an algebraically closed field and let $A$ be a $k$-algebra. 
If there exists a finite-dimensional filtration $\{V_n \colon n\in\NN\}$ on $A$ such that the associated graded algebra $B={\rm gr}_V(A)$ is a (not necessarily commutative) domain over $k$, then $A$ is a finite factorization domain over $k$.\\
\label{thm: main}
\end{theorem}
An immediate corollary of this is that an $\mathbb{N}$-graded domain over an algebraically closed field whose homogeneous components are finite-dimensional is an FFD.  This may appear to be surprising since there are examples of graded domains that are not FFDs---these examples, however, all have some infinite-dimensional homogeneous component.   In general, when one does not work over an algebraically closed base field, one can still in practice use this result, if one has that the algebras involved behave well under extension of scalars.  A useful corollary is then the following result.
\begin{corollary}
Let $k$ be a field and let $A$ be a $k$-algebra. 
If there exists a finite-dimensional filtration $\{V_n \colon n\in\NN\}$ on $A$ such that the associated graded algebra $B={\rm gr}_V(A)$ has the property that $B\otimes_k \bar{k}$ is a (not necessarily commutative) domain, then $A$ is a finite factorization domain.\\
\label{thm: main2}
\end{corollary}

In the body of the paper, we will show that an extended version of
Theorem \ref{thm: main} holds (see Remark \ref{rem: monoid}) that
considers more general filtrations.
Although our main
result has what might appear to be a technical condition on the
algebra having a filtration with the property that the associated
graded algebra is a domain, rings of this sort are actually ubiquitous
in both noncommutative and commutative algebra and thus Theorem
\ref{thm: main} applies to a large class of algebras.  For example, if
one takes the first Weyl algebra $A_1$ over a field $k$ of characteristic
zero, described above, then we have an $\mathbb{N}$-filtration given
by taking $V_n$ to be the span of all monomials of the form $x^i
\partial^j$ with $i+j\le n$.  If one then forms the associated graded
algebra, one gets the polynomial ring $k[u,v]$, which is a
domain whose $n$-th homogeneous piece is the span of all elements of
the form $u^iv^j$ with $i+j=n$.  In particular, this theorem
immediately recovers the result that $A_1$ over $k$ is a finite
factorization domain.  More generally, we obtain the following theorem
as corollary,
showing that a broad class of so-called $G$-algebras (see Def. \ref{def: G-algebra}) as well as their subalgebras (see \S 2 for definitions of these algebras and their associated filtrations)
consists of finite factorization domains.
\begin{theorem}
Let $k$ be a field. Then $G$-algebras over $k$ and their subalgebras are finite factorization domains.
In particular, so are
\begin{enumerate}
\item the Weyl algebras and the shift algebras;\label{itm:weylShift}
\item enveloping algebras of finite-dimensional Lie algebras;\label{itm:envLie}
\item coordinate rings of quantum affine spaces;\label{itm:qntmAffineSpaces}
\item $q$-shift algebras and $q$-Weyl algebras;\label{itm:qWeylqShift}
\end{enumerate}
as well as polynomial rings over the algebras listed in items (\ref{itm:weylShift})--(\ref{itm:qWeylqShift}).
\label{thm: application}
\end{theorem}

In addition, we can give an upper bound on the number of factorizations of elements in our ring $A$.  To explain our result, we must first define \emph{growth functions}.
Given a field $k$ and a $k$-algebra $A$ with an associated filtration $V=\{V_n\colon n\in\NN\}$, we define the growth function of $A$ with respect to $V$ via the rule
 \begin{equation}
 \label{eq: V}
 g_V(n) := \dim_k(V_n).
 \end{equation}
With this notion fixed, we can actually give upper bounds on the number of factorizations of elements.
\begin{theorem}
  \label{thm: enumeration}
  Let $k$ be an algebraically closed field and let $A$ be a $k$-algebra with a filtration $V=\{V_n\}$ such that the associated graded algebra of $A$ with respect to $V$ is a domain. Then an element $a \in V_{n}$ has at most $$\frac{n}{4}\cdot 4^{g_V(n)}$$ distinct factorizations into two elements and at most 
  $$2^{n\cdot g_V(n)}$$ total distinct factorizations up to multiplication of factors by central units.
\end{theorem}
These bounds are most likely far from optimal and it would be interesting whether a significant improvement could be found.

Given that many ring theoretic properties can be ``lifted'' from the
associated graded ring, it is tempting to expect that an algebra with
an associated graded ring that is a finite factorization domain should
itself be a finite factorization domain.  When the base field is
algebraically closed, Theorem \ref{thm: main} gives in fact a much
stronger result: one only needs that the associated graded ring is a
domain! On the other hand, if $k$ is not assumed to be algebraically
  closed then we do not know whether there are non-FFDs whose
  associated graded rings are finite factorization domains.

The outline of the paper is as follows.  In \S2 we give an overview of the classes of algebras to which Theorem \ref{thm: main} applies.  In \S3 we prove Theorems \ref{thm: main} and \ref{thm: enumeration} under the common Theorem \ref{thrm:asgrufd}.  In \S4 we prove Corollary \ref{thm: main2} and then give applications of  this corollary and prove Theorem \ref{thm: application} in \S5.  In \S6 we give some examples of non-FFDs.

\subsection{Notation}
Throughout this paper $\NN$ denotes the natural numbers including zero, 
$k$ denotes a field, $\bar{k}$ the algebraic closure of $k$,  $\delta_{i,j}$ the Kronecker delta function, $[a,b]=ab-ba$ the Lie bracket.

\section{Examples and associated filtrations}

Our results are mainly based on the ring-theoretic properties of the
associated graded algebra of certain classes of noncommutative algebras. 
In what follows we will introduce several algebras of
interest and the properties of their associated graded algebras. For a
more thorough introduction to associated graded algebras, we refer the
reader to the book of McConnell and Robson
\citep{mcconnell2001noncommutative}.

  \begin{definition}
Let $k$ be a field and let $n$ be a positive integer. The $n$-\emph{th Weyl algebra}, $A_{n}$, is the unital $k$-algebra with generators
$ x_1, \ldots, x_n ,\partial_1, \ldots, \partial_n$ subject to the relations $[x_i,x_j]=[\partial_i,\partial_j]=0$ and $[\partial_i,x_j]=\delta_{i,j}$ for $i,j\in \{1,\ldots ,n\}$.
    As a notational convenience,
  we write
  \[
  {\bf x}^{\bf e} {\bf \d}^{\bf w} := x_1^{e_1}\cdots x_n^{e_n}\partial_1^{w_1} \cdots
  \partial_n^{w_n},
  \]
  where ${\bf e}=(e_1,\ldots ,e_n)$ and ${\bf w}=(w_1,\ldots ,w_n)$ are $n$-tuples of nonnegative integers.  
\end{definition}
A classical result states, that the elements of the form ${\bf x}^{\bf e} {\bf \d}^{\bf w}$ with ${\bf e}, {\bf w} \in \NN^n$
  form a $k$-basis for the $A_n$.   We also point out that the $n$-th
  Weyl algebra can be seen as the algebra of linear partial
  differential operators on the affine $n$-space in the case when $k$ is a field of characteristic zero.  
    We note that if $\bar{k}$ is the algebraic closure of $k$ then $A_n\otimes_k \bar{k}$ is isomorphic to the $n$-th Weyl algebra over $\bar{k}$.
\begin{example}
\label{ex:WeylFilter}
 {\em We define a \textit{total degree}  $\deg(a)$
of any element $a \in A_n$ as follows: by convention, $\deg(0) = -\infty$. For a monomial ${\bf x}^{\bf e} {\bf \d}^{\bf w}$, with ${\bf e}=(e_1,\ldots ,e_n) \in \NN^n$ and ${\bf w}=(w_1,\ldots ,w_n)\in \NN^n$, the total degree is defined as $\sum_{i =
  1}^n (e_i + w_i)$.  The degree of a general nonzero element $f$ of $A_n$ is then determined by writing $f$ (uniquely) as a linear combination of monomials and defining its total degree to be the maximum of the total degrees of the monomials that appear in the linear combination with a nonzero coefficient.  
We note that the total degree then induces a filtration on $A_n$, namely
$V=\{V_d \colon d \in\NN \}$, where $V_d := \{f \in A_n \mid \deg(f) \leq d\}$.

The associated graded algebra of $A$ with respect to $V$ is easily seen to be 
isomorphic to a commutative multivariate polynomial ring in $2n$ variables.}
\end{example}

\begin{definition}
Let $A$ be a $k$-algebra, generated by $x_1, \ldots, x_n$ subject to some relations. It is said that $A$ has a Poincar\'e-Birkhoff-Witt (shortly PBW) basis, if $\{ x_1^{a_1}\cdot \ldots \cdot x_n^{a_n}\colon a_i \in \NN \}$ is a $k$-basis of $A$.
\end{definition}

It is obvious, that a commutative ring $k[x_1, \ldots, x_n]$ has a PBW basis. 
On the contrary, the free associative algebra $k\langle x_1, \ldots, x_n \rangle$ does not possess a PBW basis, since its canonical $k$-basis is formed by all elements from the free monoid $\langle x_1, \ldots, x_n \rangle$.

Note, that in general the PBW property should be treated with respect to the
given order of variables $x_1, \ldots, x_n$. 

Many important noncommutative algebras possess PBW bases. It is the case for universal enveloping algebras of finite-dimensional Lie algebras, for certain quantum groups (among other, from Def. \ref{def:quantumAlg}, \ref{def:qWeyl}) and for many algebras, associated to linear operators (partial differentiation/difference/integration as well as their $q$-analogues, see Corollary \ref{cor: OpAlg}) to name a few.  Related to this, we give the notion of a $G$-algebra.

\begin{definition}
\label{def: G-algebra}
For $n\in\NN$ and $1\leq i < j \leq n$ consider the units $c_{ij} \in k^*$ and 
polynomials $d_{ij} \in k[x_1,\ldots,x_n]$. Suppose, that there exists a
monomial total  well-ordering $\prec$ on $k[x_1,\ldots,x_n]$, such that 
for any $1\leq i < j \leq n$ either $d_{ij}=0$ or the leading monomial of $d_{ij}$ is smaller than $x_i x_j$ with respect to $\prec$. The $k$-algebra
$A := k\langle x_1,\ldots, x_n \mid \{ x_j x_i = c_{ij} x_i x_j + d_{ij} \colon  1\leq i < j \leq n \} \rangle$ is called a $G$-\emph{algebra}, if $\{x^{\alpha} : \alpha\in\NN^n \}$ is a $k$-basis of $A$.
\end{definition}
$G$-algebras \citep{AP, LS03, LVDiss} are also known as algebras of solvable type \citep{KW, HLi} and as PBW algebras \citep{GomezRefiltering, Bueso:2003}. See \citep{G14}
for a comprehensive source on these algebras.  
We note that among algebras possessing a PBW basis,
  $G$-algebras form a proper subset.
For example, if one takes the algebra $A$ generated by $x_1,x_2,x_3$
with relations $x_2x_1=x_1x_2x_3$ and $[x_3,x_1]=[x_3,x_2]=0$ then it
is straightforward to show that $\{x_1^a x_2^b x_3^c \colon
a,b,c\in\NN\}$ forms a $k$-basis for $A$.  On the other hand, $A$
cannot be a $G$-algebra with respect to any ordering because the
Gelfand-Kirillov dimension of $A$ is $4$ and a $G$-algebra on $n$
variables necessarily has Gelfand-Kirillov dimension $n$.  We note
that $G$-algebras are Auslander-regular Noetherian domains
\citep[Theorem 3 and Proposition 4]{GomezAuslander}
and if $A$ and $B$ are $G$-algebras over $k$ then $A \otimes_k B$ is again a $G$-algebra \citep{LVDiss}.

\begin{definition}
\label{def:quantumAlg}
 Let $k$ be a field.  Let us fix a collection of nonzero elements ${\bf q}=(q_{i,j})_{1\le i,j\le n}$ of $k$ satisfying $q_{i,j}q_{j,i}=1$ for $i\neq j$ and $q_{i,i}=1$. 
  A \emph{coordinate ring of the quantum affine} $n$-\emph{space} $\mathcal{O}_{\bf q}(k^n)$ for $n \in \NN$ is a $k$-algebra, generated by $x_1,\ldots, x_n$ subject to relations $x_i x_j = q_{i,j} x_j x_i$ for $1 \leq i,j\leq n$.  These algebras can be viewed as quantizations of the coordinate ring of $k^n$, which is the reason for having quantum affine space in the name.
\end{definition}

\begin{remark}
As one can see from the relations, the algebras $\mathcal{O}_{\bf q}(k^n)$ are naturally graded, like commutative rings $k[x_1,\ldots, x_n]$, by
any tuple $w_1,\ldots,w_n \in \ZZ_{+}$. On the other hand, 
there is a filtration coming from taking $V_{n}$ to be the vector space spanned by the monomials of the form $x_1^{a_1}\cdots x_n^{a_n}$ with $\sum a_j \le n$.  Constructing the associated graded algebra results in a ring structure which is isomorphic to the original ring.  Since the original ring is a domain, we see that the hypotheses given in Theorem \ref{thm: main} are satisfied.
\end{remark}


\begin{definition}
  \label{def:qWeyl}
Let us consider an $n$-tuple ${\bf q}=(q_1,\ldots ,q_n)\in (k^*)^n$. 
  The \textit{$n$-th $q$-Weyl algebra} $Q_{n}$ for $n \in \NN$ is the
  $k$-algebra with generators $x_1,\ldots ,x_n,
  x_{n+1}:=\partial_1,\ldots ,x_{2n}:=\partial_n$ and relations
  $\partial_ix_j = (1+\delta_{i,j}(q_i-1)) x_j \partial_i +
  \delta_{i,j}$ and $[x_i, x_j]=[\partial_i,\partial_j]=0$ for $i,j\in
  \{1,\ldots ,n\}$.  We note that by setting $q_1=\cdots =q_n=1$ one
  recovers the classical Weyl algebra. Just as with the $n$-th Weyl algebra $A_n$,  we write
  \[
  {\bf x}^{\bf e} {\bf \d}^{\bf w} := x_1^{e_1}\cdots x_n^{e_n}\partial_1^{w_1} \cdots
  \partial_n^{w_n},
  \]
  where ${\bf e}=(e_1,\ldots ,e_n)$ and ${\bf w}=(w_1,\ldots ,w_n)$ are $n$-tuples of nonnegative integers.\end{definition}

\begin{remark}
 $Q_n$ possesses a PBW basis. One can define the filtration and the associated graded algebra of $Q_n$ also with respect to the total degree of the elements in $Q_n$. However, in contrast to $A_n$, the ring $\grdr_V(Q_n)$ will
 not be commutative. As one can easily see, we obtain a coordinate ring of the quantum affine space of dimension $2n$ with parameters $q_{i,j} = 1$ except for $q_{1,n+1}=q_1, \ldots, q_{n,2n}=q_n$.
 \end{remark}

\section{Proof of Theorems \ref{thm: main} and \ref{thm: enumeration}}
In this section, we prove Theorem \ref{thm: main}.  
We make use of several notions from algebraic geometry, 
cf. the books of Hartshorne \citep{hartshorne1977algebraic} and Cox, Little, and O'Shea \citep{cox1992ideals}.


\begin{definition}
\label{def: FFD}
  Let $A$ be a (not necessarily commutative) domain.  We say that $A$ is a
  \textit{finite factorization domain} (FFD, for short), if
  every nonzero, non-unit element of $A$ has at least one
  factorization into irreducible elements and there are at most
  finitely many distinct factorizations into irreducible elements up
  to multiplication of the irreducible factors by central units in $A$.
  \end{definition}

We remark that some definitions 
found in the literature count factorizations in which the factors are multiplied in different orders as being the same.  We note that removing this does not affect finiteness, since there are only finitely many permutations of factors and hence the definition above is equivalent to the usual definition in the commutative setting.  In most examples that we will look at, we will be working with a domain $A$ that is a finitely generated algebra over some field $k$ and such that the central units in $A$ are precisely the elements of $k^*$.

    Our key result in obtaining finiteness of factorizations is Proposition \ref{lem: vs} below.  In order to prove this, we require a basic result \cite[Lemma 1.1]{RowenSaltman2013}, which we reprove below for the convenience of the reader.

\begin{lemma}
\label{lem: val} Let $k$ be an algebraically closed field and let $A$ be a finitely generated $k$-algebra that is a domain.  If $K$ is a field extension of $k$ then $A\otimes_k K$ is a domain.
\end{lemma}
\begin{proof}
Suppose not.  Then we can find nonzero $x=\sum_{i=1}^p a_i\otimes \alpha_i$ and $y=\sum_{i=1}^q b_i \otimes \beta_i$ in $A\otimes_k K$ such that $xy=0$.  By picking expressions for $x$ and $y$ with $p$ and $q$ minimal, we may assume that $a_1,\ldots ,a_p$ are linearly independent over $k$ and that $b_1,\ldots ,b_q$ are linearly independent over $k$ and that the $\alpha_i$ and $\beta_j$ are all nonzero.
Let $R$ denote the finitely generated $k$-subalgebra of $K$ generated by $\alpha_1,\ldots ,\alpha_p,\beta_1,\ldots ,\beta_q$ and $\alpha_1^{-1}$ and $\beta_1^{-1}$.  Then we may regard $x$ and $y$ as nonzero elements of $A\otimes_k R$ such that $xy=0$.  Let $P$ be a maximal ideal of $R$.  We have a 
$k$-algebra homomorphism $\phi: A\otimes_k R\to A\otimes_k (R/P)$ given by $\sum a_i\otimes r_i \mapsto \sum a_i \otimes \pi(r_i)$, where $\pi: R\to R/P$ is the canonical surjection.  Observe that $\pi(\alpha_1)$ and $\pi(\beta_1)$ are nonzero since $\alpha_1$ and $\beta_1$ are units in $R$ and thus $x':=\phi(x)=\sum a_i \otimes \pi(\alpha_i)$ is nonzero, as $a_1,\ldots ,a_p$ are linearly independent over $k$.  Similarly, $y':=\phi(y)$ is nonzero.  Thus $x'y'=\phi(xy)=\phi(0)=0$ in $A\otimes_k (R/P)$. But $R/P\cong k$ by the Nullstellensatz and hence $A\otimes_k (R/P)\cong A\otimes_k k= A$, which is a domain and so we cannot have $0=x'y'$ with $x'$ and $y'$ nonzero.   The result follows.
\end{proof}
\begin{proposition} Let $k$ be an algebraically closed field, let $A$ be a $k$-algebra that is a domain, and let $V$ and $W$ be finite-dimensional $k$-vector subspaces of $A$.  If $a\in A$ is nonzero then, up to multiplication by elements of $k^*$, there are less than $2^{{\rm dim}(V)+{\rm dim}(W)}$ distinct factorizations of the form $a=bc$ where $b\in V, c\in W$.
\label{lem: vs}
\end{proposition}
\begin{proof}
Suppose that this is not the case.  Let $v_1,\ldots, v_p$ be a basis for $V$, let $w_1,\ldots ,w_q$
be a basis for $W$, and let $u_1,\ldots ,u_t$ be a
basis for the space $V \cdot W$.  Then we have $v_i w_j = \sum_{\ell=1}^t
\gamma_{i,j,\ell} u_{\ell}$ for some $\gamma_{i,j,\ell}\in k$.  If $a\not\in VW$ then the result holds vacuously.  Thus we assume that $a\in VW$ and we write 
$a=\sum \alpha_i u_i$ with
$\alpha_1,\ldots ,\alpha_t\in k$.  Then a factorization $a=bc$ with
$b\in V$ and $c\in W$ just corresponds to a solution
$(x_1,\ldots ,x_p,y_1,\ldots ,y_q)\in k^{p+q}$ to the equation
$$\left( \sum_i x_i v_i\right)\left( \sum_j y_j w_j\right) = \sum \alpha_i u_i.$$
Since we only care about factorizations up to multiplication of the factors by scalars from $k^*$, it suffices to consider factorizations in which
$x_i=1$ for some $i\in \{1,\ldots, p\}$.  

Expanding this product out, we get a system of equations
\begin{equation}
\label{eq: var}
\sum_{i,j} \gamma_{i,j,\ell} x_i y_j = \alpha_{\ell} 
\end{equation}
for $\ell=1,\ldots ,t$.  We now let $X_i$ denote the subvariety
of $\mathbb{A}_{k}^{p+q}$ given by the zero set of the quadratic
equations in Equation (\ref{eq: var}) and the equation $x_i=1$.

By assumption, there are infinitely many distinct factorizations and hence there is some $s\in \{1,\ldots ,p\}$ such that $X_s$ is infinite.  In particular, $X_s$ has some positive-dimensional irreducible component $Z\subseteq \mathbb{A}_{k}^{p+q}$.

Let $\pi_{x_i} :\mathbb{A}_{k}^{p+q}\to \mathbb{A}_{k}^1$
be the projection $(x_1,\ldots ,x_p,y_1,\ldots ,y_q)\mapsto x_i$ for
$i=1,\ldots ,p$.  Similarly, we define the projections $\pi_{y_j}$ for
$j=1,\ldots ,q$.  Then the restriction maps $f_i:=\pi_{x_i}|_Z$ and
$g_j:=\pi_{y_j}|_Z$ are regular maps to $k$ and hence they 
can be viewed as
elements of the ring $k[Z]$ of regular functions of $Z$.  Furthermore, by
construction we have that $f_s$ is the constant function $f_s=1$.  A point $z \in Z$ gives a factorization
\begin{equation} \label{abcz}
a=b_{z}c_{z}
\end{equation} via the
rule
$$b_{z} = \sum f_i(z) v_i\qquad c_{z} = \sum g_j(z)w_j.$$  Furthermore, since the points of $X_s$ give rise to distinct
factorizations of $a$, we see that the tuples $\{(b_{z},c_{z})\colon
z\in Z\}$ are all distinct---even up to multiplication by constants.   We claim that the collection of 
elements $c_z$, with $z\in Z$ are distinct in $\mathbb{P}(W)$ (that is, distinct in $W$ up to multiplication by nonzero elements of $k$).  To see this, suppose that there exist distinct points $z_1$ and $z_2$ of $Z$ such that $c_{z_1}=\lambda c_{z_2}$ for some $\lambda\in k$.  Then we have
$$a=b_{z_1}c_{z_1} = \lambda b_{z_1}c_{z_2} = b_{z_2}c_{z_2}$$ and hence
$(b_{z_2}-\lambda b_{z_1})c_{z_2}=0$.  Since $A$ is a domain and $c_{z_2}$ is a nonzero element of $A$, we see that
$b_{z_2}=\lambda b_{z_1}$ and so the factorizations $a=b_{z_1}c_{z_1}=b_{z_2}c_{z_2}$ are the same up to multiplication by scalars, which is a contradiction. 
In particular, there must be some $h\in \{g_1,\ldots ,g_q\}$ that is non-constant, and so by reindexing if necessary, we may assume that $g_1$ is a non-constant regular map on $Z$.

We now express the identity from Equation (\ref{abcz}) in a larger algebra.  Consider the algebra $B:=A\otimes_k k[Z]$, where the elements of the form $1\otimes \alpha$ with $\alpha\in k[Z]$ are central.
Then inside of $B$ we can find the elements $${\bf b}:=\sum_i v_i\otimes f_i\qquad {\rm and}\qquad {\bf c}=\sum_j w_j\otimes g_j.$$
We claim that $a\otimes 1 = {\bf b}{\bf c}\in B$.  To see this, observe that given a point $z\in Z$, we have a specialization map ${\bf e}_z: B\to A$ given by
$${\bf e}_z\left( \sum_{i=1}^d a_i \otimes h_i\right) = \sum_{i=1}^d h_i(z) a_i,$$ where we note that $h_1(z),\ldots ,h_d(z)\in k$ and so the output is indeed in $A$.  Furthermore, it is straightforward to check that this map is well-defined and that it is a $k$-algebra homomorphism.  Moreover, by construction ${\bf e}_z({\bf b}) = b_z$ and ${\bf e}_z({\bf c})=c_z$ and so ${\bf e}_z({\bf b}{\bf c})=a = {\bf e}_z(a\otimes 1)$ for all $z\in Z$.  In particular,
${\bf e}_z(a\otimes 1- {\bf b}{\bf c})=0$ for every $z\in Z$.
If $a\otimes 1 - {\bf b}{\bf c}$ is nonzero then we may write $a\otimes 1 - {\bf b}{\bf c} = \sum_{i=1}^d a_i \otimes h_i$, for some $d\ge 1$.  By taking $d$ minimal, we may assume that $a_1,\ldots ,a_d\in A$ are linearly independent over $k$ and that $h_1,\ldots ,h_d\in k[Z]$ are linearly independent over $k$.  Then since $h_1,\ldots ,h_d$ are linearly independent regular maps, they are not identically zero.  In particular, there is some point $z\in Z$ such that $h_1(z)\neq 0$.  We now apply the map ${\bf e}_z$ to obtain:
$$0 = {\bf e}_z(a\otimes 1 - {\bf b}{\bf c}) = \sum_{i=1}^d h_i(z) a_i,$$ which gives a contradiction since $a_1,\ldots ,a_d$ are linearly independent over $k$ and $h_1(z),\ldots ,h_d(z)\in k$ with $h_1(z)$ nonzero.  Thus we see that we have the desired factorization
$$a\otimes 1 = {\bf b}{\bf c}$$ in the algebra $B$.  We note that $B$ is a subalgebra of the algebra $C=A\otimes_k k(Z)$ and we may regard the factorization $a\otimes 1 = {\bf b}{\bf c}$ as a factorization in $C$.

Now since $g_1\in k[Z]$ is non-constant and $k$ is algebraically closed, we have that $1/g_1$ is transcendental over $k$.  Since $k(Z)$ is finitely generated over $k$ it has some finite transcendence degree $\kappa$.  We may then extend $1/g_1$ to a transcendence base
$h_1=1/g_1,h_2,\ldots ,h_{\kappa}$ for $k(Z)$ over $k$; that is, $h_1,\ldots ,h_{\kappa}$ are algebraically independent over $k$ and $k(Z)$ is an algebraic extension of the subfield $F:=k(h_1,\ldots ,h_{\kappa})$.  Moreover, since $k(Z)$ is finitely generated over $k$ we have that $k(Z)$ is a finite extension of $F$.  

We let $R$ denote the $k$-subalgebra of $k(Z)$ generated by $h_1,\ldots ,h_{\kappa}$ and we let $S$ denote the integral closure of $R$ in $k(Z)$. 
We claim that the field of fractions of $S$ is equal to $k(Z)$.  To see this, we observe that since $k(Z)$ is algebraic over $F$ and since $F$ is a field we have that $k(Z)$ is equal to the integral closure of $F$ in $k(Z)$.  Now $S$ is the integral closure of $R$ in $k(Z)$ and so if we take $U$ to be the multiplicatively closed subset $R\setminus \{0\}$ then we have that $U^{-1}S$ is equal to the integral closure of $U^{-1}R$ in $U^{-1} k(Z)=k(Z)$ \cite[Proposition 4.13]{eisenbud1995commutative}.  But $U^{-1}R=F$ and since the integral closure of $F$ in $k(Z)$ is equal to $k(Z)$, we see that $U^{-1}S=k(Z)$.  Since $k(Z)=U^{-1}S\subseteq {\rm Frac}(S)\subseteq k(Z)$, we obtain the claim.

 Next, observe that $h_1$ generates a height one prime ideal of $R$ and so by going up (see \cite[Proposition 4.15]{eisenbud1995commutative}), there is a height one prime ideal $P$ of $S$ such that $P\cap R=R h_1$.  Since $S$ is integrally closed, we have that $S_P$, the localization of $S$ at the complement of the prime ideal $P$, is a discrete valuation ring with valuation $\nu_0: S_P\setminus \{0\}\to \mathbb{Z}$ given by $\nu_0(s)=m$ where $m$ is the unique nonnegative integer such that $s\in P_P^m$ but $s\not\in P_P^{m+1}$ (cf. \cite[Theorem 11.2]{eisenbud1995commutative}).  This valuation extends to a valuation $\nu: {\rm Frac}(S_P)\setminus \{0\}=k(Z)^* \to \mathbb{Z}$ via the rule $\nu(a/b)=\nu_0(a)-\nu_0(b)$ for $a,b\in S_P\setminus \{0\}$.  By Nakayama's lemma (see \cite[Corollary 4.8]{eisenbud1995commutative}) there exists some $\theta\in PS_P\setminus (PS_P)^2$ and thus $\nu(\theta)=1$.  
 
Now we let $$m_1=\max \{-\nu(f_1),\ldots
,-\nu(f_p)\}$$ and let $$m_2=\max \{-\nu(g_1),\ldots, -\nu(g_q)\}.$$
Then $m_2\ge - \nu(g_1)>0$ since $h_1=1/g_1\in PS_P$ and so $-\nu(g_1)=\nu_0(h_1)>0$.  
Also $m_1 \ge 0$, since $f_s$ is a nonzero constant.  Then multiplying both sides of the equation
$a\otimes 1 = {\bf b}{\bf c}\in C$ by $1\otimes \theta^{m_1+m_2}$, which is central in $C$, we see that
$$a\otimes \theta^{m_1+m_2} = \left(\sum_i v_i \otimes (f_i \theta^{m_1})\right)\left( \sum_j w_j \otimes (g_j \theta^{m_2})\right).$$
Let $T$ denote the subring of $k(Z)$ consisting of all rational functions $h$ for which $\nu(h)\ge 0$.  Then $T$ is a local ring with maximal ideal $Q$ consisting of all elements with $\nu(h)>0$.  Then by construction
 $\theta^{m_1+m_2},f_1\theta^{m_1},\ldots ,f_p\theta^{m_1},g_1\theta^{m_2},\ldots ,g_q\theta^{m_2}\in T$, and so the factorization 
$$a\otimes \theta^{m_1+m_2} = \left(\sum_i v_i \otimes (f_i \theta^{m_1})\right)\left( \sum_j w_j \otimes (g_j \theta^{m_2})\right)$$ holds in the subalgebra $A\otimes_k T$ of $C$.  
Now let $K=T/Q$ and let $\pi:T\to K$ be the canonical homomorphism.  Then $K$ is a field extension of $k$ and we have a $k$-algebra homomorphism 
$\phi : A\otimes_k T\to A\otimes_k K$ given by
$$\sum a_i \otimes t_i \mapsto a_i \otimes \pi(t_i).$$
Since $m_1+m_2>0$ we see that $\theta^{m_1+m_2}\in Q$ and so 
$\phi(a\otimes \theta^{m_1+m_2})=0$.  On the other hand,
$$b':=\phi\left(\sum_i v_i \otimes (f_i \theta^{m_1})\right)$$ is nonzero since the $v_i$ are linearly independent over $k$ and by construction there is some $\ell$ such that $\pi(f_{\ell}\theta^{m_1})\in K$ is nonzero.  Similarly, 
$$c':=\phi\left( \sum_j w_j \otimes (g_j \theta^{m_2})\right)$$ is nonzero.  Thus we have $0=b'c'$ in $A\otimes_k K$ with $b'$ and $c'$ nonzero.  But this is impossible by Lemma \ref{lem: val} since $A$ is a domain.

It only remains to estimate the size of the union of the $X_i$.  From
Equation (\ref{eq: var}) it follows that each $X_i$ is the zero set of $t$
quadratic polynomials along with the linear polynomial $x_i=1$ in
$\mathbb{A}_k^{p+1-i+q}$.  By using the linear relation $x_i=1$ to
reduce the dimension of the ambient affine space, we may regard each
$X_i$ as being the zero set of $t$ quadratic polynomials in
$\mathbb{A}_k^{p-i+q}$.  Thus by a general B\'ezout inequality
(e.~g. \citep[Theorem 3.1]{schmid1995affine}) we see that since
each $X_i$ is finite, it must have size at most $2^{p-i+q}$.  Since
the union of the $X_i$ as $i$ ranges from $1$ to $p$ gives the total
number of distinct factorizations, we see that the total number is at
most $$\sum_{i=1}^p 2^{p+q-i} = 2^{p+q} - 2^q < 2^{p+q}.$$ The result
now follows by observing that $p={\rm dim}(V)$ and $q={\rm dim}(W)$.

\end{proof}

A few remarks are in order.
\begin{remark}
Suppose that $k$ is a finite field. Then Proposition \ref{lem: vs} holds trivially
for $k$. 
Indeed, if $a = bc = (\sum b_i v_i) \cdot (\sum c_j w_j)$ 
for $b_i, c_j \in k$ is a factorization, there are only finitely many choices for 
$b_i, c_j$, what gives a finite upper bound of $|k|^{p+q}$ possibilities to choose $(b,c)$.
\end{remark}
\begin{remark} 
\label{rem: vs2} We note that in general one cannot eliminate the algebraically closed hypothesis in the statement of Proposition \ref{lem: vs}.  
For example, if $k=\mathbb{R}$ and $A=\mathbb{R} + \CC[t]\cdot t\subseteq \mathbb{C}[t]$ and if we let $V=W=\mathbb{R}+\mathbb{C}t$,
then we have infinitely many factorizations of $t^2$ in $VW$ of the form
\[
t^2 = (\cos(\theta) + i\sin(\theta))t \cdot (\cos(\theta) -
i\sin(\theta))t
\]
for any  $\theta \in [0,2\pi)$. Notice that the units of $A$ are precisely the nonzero real numbers and hence for $\theta\in [0,\pi)$ these factorizations are distinct up to multiplication by central units.   
\end{remark}

We now prove Theorems \ref{thm: main} and \ref{thm: enumeration}.  In
fact, we give a slightly different statement --- easily seen to be
equivalent --- which only deals with factorizations up to nonzero scalars.  In fact, these are exactly the units of the algebra in our context, as one can observe by looking at the associated graded algebra.
\begin{theorem}
  \label{thrm:asgrufd}
Let $k$ be an algebraically closed field, and let $A$ be a $k$-algebra 
with a finite-dimensional filtration $V=\{V_n\}_{n\ge 0}$ such that the associated graded algebra is a (not necessarily commutative) domain.  
Then $A$ is an FFD. 
Moreover, if $a\in V_n$ then $a$ has at most $\frac{n}{4}\cdot 4^{g_V(n)}$ distinct factorizations into two irreducible elements and at most $2^{n\cdot g_V(n)}$ total factorizations, where we only consider factorizations up to multiplication by central units.  
\end{theorem}
Before we prove this, we remark that graded non-FFDs over algebraically closed fields do exist (see Example \ref{ex:EasyNonFFD}).  However, in all such instances we necessarily have some infinite-dimensional homogeneous component.

\begin{proof} 
We remark that $V_0$ must itself be a $k$-algebra since $V_0V_0\subseteq V_0$.  Moreover, it is a domain by hypothesis and finite-dimensional over $k$.  Since $k$ is algebraically closed, we then see that $V_0=k$.
We now show that every nonzero non-unit element of $A$ has at least one factorization into irreducibles.  To do this, suppose that there exists some element $a\in A$ that is nonzero and not a unit that does not factor into irreducibles.  Then among all such $a$ we pick one such that $a\in V_n$ with $n$ minimal.  Then $a$ is not irreducible or it would already have a factorization and hence $a=bc$ for some $b\in A$ and $c\in A$ that are non-units.  Since the associated graded ring is a domain, we see that $b\in V_i$ and $c\in V_{n-i}$ for some $i\le n$.  In fact, since $V_0=k$ and since $b$ and $c$ are non-units we have that both $i$ and $n-i$ are less than $n$, and so by minimality of $n$ we see that $b$ and $c$ have a factorization into irreducibles, which in turn gives a factorization of $A$ into irreducibles.  

We also note that the units of $A$ are precisely the elements of $k^*=V_0\setminus \{0\}$.  The reason for this is that if $u\in V_i\setminus V_{i-1}$ for some $i>0$ then for nonzero $v\in A$, the element $uv$ must be in $V_j$ for some $j\ge i$ since the associated graded ring is a domain.  Since $1\in V_0$, we see that the only units of $A$ are in $V_0$ and this gives the claim.

We finally show that $A$ is an FFD.  We have shown that it is sufficient to count factorizations up to multiplication of factors by scalars from $k^*$.
Suppose, towards a contradiction, that this is not the case and  let $n$ be the smallest element of $\mathbb{N}$ for which there is some $a\in V_{n}$ with infinitely many distinct factorizations.  We note that by minimality of $n$ it follows that $a$ has infinitely many distinct factorizations of the form $a=bc$.  We note that $V_{0}$ must be itself a domain that is finite-dimensional over $k$.  Since $k$ is algebraically closed we thus see that $V_{0}$ is equal to $k$.  In particular any proper factorization of $a$ of the form $a=bc$ must have $b\in V_{n_1}$ and $c\in V_{n_2}$ with $n_1+n_2=n$ and $n_1,n_2>0$.  We may assume that $A$ is infinite-dimensional over $k$; if it is not, there is nothing to prove since we
 then have $A=k$, because $A$ is a domain and $k=\bar{k}$.

By Proposition \ref{lem: vs} we have that the number of distinct factorizations of $a$ of the form $a=bc$ with $b\in V_{n_1}$ and $c\in V_{n_2}$ is at most
$2^{g_V(n_1)+g_V(n_2)}$.  Thus the total number of distinct factorizations of $a$ into two elements is at most 
$$\sum_{n_1+n_2=n} 2^{g_V(n_1)+g_V(n_2)}.$$
There are at most $n-1$ solutions to the equation $n_1+n_2=n$.  Moreover, $g_V(n_1)$ and $g_V(n_2)$ are both at most $g_V(n)-1$ since $n_1,n_2<n$ and $A$ is infinite-dimensional and so we see that the number of different factorizations of $a$ up to multiplication by scalars is at most $(n-1) \cdot 4^{g_V(n)-1}$ when $a\in V_{n}$.  
In particular, a simple recursion now gives that the total number of factorizations (with any number of factors) up to multiplication by scalars is at most
$$\sum_{j\ge 2} \sum_{n_1+\cdots +n_j=n} 2^{\sum_j g_V(n_j)}.$$
Since each $n_i$ appearing in the sum is nonzero, we see that this sum can in fact be expressed as
$$\sum_{j=2}^{n} \sum_{n_1+\cdots +n_j=n} 2^{\sum_j g_V(n_j)}.$$
As before, for $j\le n$ and $n_1+\cdots +n_j=n$, we have
$2^{\sum_j g_V(n_j)}\le 2^{n\cdot (g_V(n)-1)}$.  Since the number of compositions of $n$ is $2^{n}$ we obtain the upper bound
$$2^n \cdot 2^{n\cdot (g_V(n)-1)}=2^{n \cdot g_V(n)}$$ on the total number of factorizations.
\end{proof}
\begin{remark} We note that if one follows the proof, then it is easy
  to see that the FFD-part of Theorem \ref{thrm:asgrufd} holds for
  algebras whose associated graded algebras are domains when we form
  the associated graded algebra using a finite-dimensional filtration
  by an ordered monoid $(\Gamma, +)$ having the properties that \emph{any element $\gamma$ in the monoid has at most finitely many decompositions $\gamma=\gamma_1+\gamma_2$ in $\Gamma$} and that the identity element of $\Gamma$ is the unique minimal element of $\Gamma$.  We briefly give some background to explain this remark.  

We recall that a monoid $(\Gamma, +)$ is \emph{ordered} if it is endowed with a total ordering $\prec$ such that whenever $\gamma_1\prec \gamma_2$, we have $\gamma + \gamma_1\prec \gamma + \gamma_2$ and $\gamma_1 + \gamma \prec \gamma_2 + \gamma$ for all $\gamma\in \Gamma$.  

We then say that $A$ has a \emph{finite-dimensional} $\Gamma$-filtration (resp. grading) if $A$ is a union (resp. a direct sum) of finite-dimensional $k$-vector subspaces $V=\{V_{\gamma}\colon \gamma\in \Gamma\}$ such that for all $\gamma_1,\gamma_2\in \Gamma$ we have:
\begin{enumerate}
\item[(i)] $V_{\gamma_1}\subseteq V_{\gamma_2}$ whenever $\gamma_1\preceq \gamma_2$;
\item[(ii)] $V_{\gamma_1}V_{\gamma_2}\subseteq V_{\gamma_1 + \gamma_2}$.
\end{enumerate}
As before, we can form an \emph{associated graded algebra},
  denoted by ${\rm gr}^{\Gamma}_V(A)$, which is given by
$${\rm gr}^{\Gamma}_V(A) = \bigoplus_{\gamma} V_{\gamma}/V_{\prec
  \gamma},$$ where $V_{\prec \gamma}$ is $(0)$ if $\gamma$ is the
identity element in $\Gamma$ or, otherwise, $V_{\prec \gamma}$ is the
union of all $V_{\lambda}$ with $\lambda$ strictly less than $\gamma$.
One often finds it useful to form filtrations resp. gradings using the monoid
$\Gamma=\mathbb{N}^d$ with $\prec$ given by the ``weighted degree
lexicographic'' order. Namely, let the strictly positive weights $w_i$
be assigned to the components $1,\ldots,d$. Then $(e_1,\ldots
,e_d)\preceq (e_1',\ldots ,e_d')$ if and only if either $\sum w_i e_i
< \sum w_i e_i'$ or $\sum w_i e_i = \sum w_i e_i'$ and if the tuples
are not equal and $j$ is the first index $i$ for which $e_i\neq e_i'$
then $e_j < e_j'$.
An $\mathbb{N}^d$-filtration can be made into an $\mathbb{N}$-filtration by defining $V_n := \sum\{ V_{(n_1,\ldots, n_d)} : n_i \in \NN, w_1 n_1+\cdots +w_d n_d\le n \}$. The associated graded algebra will have
the same desired properties as the one with the filtration with
respect to $\mathbb{N}^d$. Hence, as far as Theorem \ref{thrm:asgrufd} is concerned, for proving that certain algebras are FFDs one can take either an $\mathbb{N}$-filtration or a finer $\mathbb{N}^d$-filtration, which is often more convenient.

\label{rem: monoid}
\end{remark}
\begin{corollary}
\label{graded}
Let $k$ be an algebraically closed field and let $$B=\bigoplus_{n\ge 0} B_n$$ be an $\mathbb{N}$-graded $k$-algebra that is a domain with the property that $B_n$ is finite-dimensional for every $n$.  Then $B$ is an FFD.
\end{corollary}
\begin{proof}
We note that $B$ has a natural filtration $V=\{V_n\colon n\in \NN\}$ coming from the grading, where $V_n=\oplus_{i\le n} B_i$.  The associated graded algebra is then isomorphic to $B$ and so the result follows from Theorem \ref{thm: main}.
\end{proof}

For counterexamples in the case that $k\neq \bar{k}$, see Example \ref{ex:EasyNonFFD} and Remark \ref{rem: vs2}.

\section{Proof of Corollary \ref{thm: main2}}
In this section, we give the proof of Corollary \ref{thm: main2}.  This is our main criterion, which we will use to give a broad class of examples of FFDs over non-algebraically closed fields.

\begin{proof}[Proof of Corollary \ref{thm: main2}]
Let $B$ denote the associated graded algebra of $A$ with respect to the filtration $\{V_n\}$.  Then $\{V_n\otimes_k \bar{k}\}$ is a filtration of $A\otimes_k \bar{k}$ and the associated graded algebra is naturally isomorphic to $B\otimes_k \bar{k}$. Then $B\otimes_k \bar{k}$ is a graded $\bar{k}$-algebra, and by hypothesis $B\otimes_k \bar{k}$ is a domain.  It follows that $A\otimes_k \bar{k}$ is an FFD.  We now claim that $A$ is an FFD.  

We note that $A$ embeds in $A\otimes_k \bar{k}$ via the homomorphism $a\mapsto a\otimes 1$ and hence $A$ is a domain as it is a subring of an FFD.  We now show that $A$ is an FFD.  Suppose that some nonzero non-unit element of $a\in A$ has infinitely many distinct factorizations.   Then $a$ has infinitely many distinct factorizations (up to multiplication by central units) into two elements.  For each such factorization $a=pq$, we have that $a\otimes 1 = (p\otimes 1)(q\otimes 1)$ gives a factorization of $a\otimes 1$ in $A\otimes_k \bar{k}$.  We note that the units of $A\otimes_k \bar{k}$ are precisely $\bar{k}^*$---this can be seen using the same argument that was employed in the proof of Theorem \ref{thrm:asgrufd}.  Since $A\otimes_k \bar{k}$ is an FFD, we necessarily have that there exist two factorizations $a=p_1q_2=p_2q_2$ such that $p_1$ is not a unit of $A$ multiplied by $p_2$ and such that $p_1\otimes 1$ and $p_2\otimes 1$ are $\bar{k}^*$-multiples of each other.  In particular, we have
$(p_1\otimes c)=p_2\otimes 1$ for some $c\in \bar{k}^*$.  But we have that $A\otimes_k \bar{k}$ is a free $A$-module with basis given by $\{1\otimes \alpha\}$ as $\alpha$ runs over a $k$-basis for $\bar{k}$.  In particular, if $p_1\otimes c=p_2\otimes 1$ then we must have $c\in k$ and so $p_1c=p_2$, a contradiction.  The result follows.  
\end{proof}
\begin{remark}
We just showed that if $A\otimes_k \bar{k}$ is an FFD then $A$ is an FFD.  In general if $A$ is an FFD then it does not in general follow that $A\otimes_k \bar{k}$ is an FFD.  For example, if one takes takes $A=\mathbb{Q}[\sqrt{2}]$ then $A$ is an FFD (in fact it is a UFD).  But $A\otimes_{\mathbb{Q}} \mathbb{Q}(\sqrt{2})$ is not since $\sqrt{2}\otimes 1 - 1\otimes \sqrt{2}\neq 0$ is a zero divisor in the tensor product, since $0=(\sqrt{2}\otimes 1 - 1\otimes \sqrt{2})(\sqrt{2}\otimes 1 + 1\otimes \sqrt{2})$. 
Even in the case when $A$ is an FFD and $A\otimes_k \bar{k}$ is a domain, we need not have that $A\otimes_k \bar{k}$ is an FFD.  As an example, if one considers Example \ref{rem: RatWeyl}, the factorizations are enumerated by the elements of the field.  If one begins with a finite base field $k$, then the resulting algebra will necessarily be an FFD.  If, however, we extend scalars up to $\bar{k}$ then the resulting algebra will still be a domain but will not be an FFD since $\bar{k}$ is infinite and thus there exist elements with infinitely many distinct factorizations.  It is interesting to know whether such behaviour can happen when one works with an infinite base field $k$.  
\end{remark}

\section{The broad class of algebras}
In this section, we give a proof of the applications claimed in Theorem \ref{thm: main}.  To complete this proof, we require a lemma that enables us to show under reasonable conditions that certain subalgebras of FFDs are FFDs.  

\begin{lemma}
Let $k$ be a field, let $A$ and $B$ be $k$-algebras with $A\subseteq B$.
If $B$ has a finite-dimensional filtration $V$, such that $\grdr_{V}(B)\otimes_k \bar{k}$
is a domain, then $A$ is a finite factorization domain over $k$.
\label{subring}
\end{lemma}
\begin{proof}
Define a filtration $W$ on $A$ as follows: for $n\in\NN$, let $W_n := V_n \cap A$.  Since $V_n$ is finite-dimensional $k$-vector space, $W_n$ is finite-dimensional
as well. Moreover, let $v \in W_i, w \in W_j$, then $vw \in W_{i+j}$ since 
$v \in V_{i} \cap A$, $w \in V_{j} \cap A$ and thus $vw \in V_{i+j} \cap A$. Moreover, the induced map from $\grdr_{W}(A)$ to $\grdr_{V}(B)$ is injective. Since $\bar{k}$ is flat as a $k$-module, we then have an injection of $\bar{k}$-algebras 
$ \grdr_{W}(A)\otimes_k \bar{k}$ into $\grdr_{V}(B)\otimes_k \bar{k}$.  By assumption, $ \grdr_{V}(B)\otimes_k \bar{k}$ is a domain and hence $ \grdr_{W}(A)\otimes_k \bar{k}$ is a domain and hence by Theorem \ref{thm: main2} we have that $A$ is an FFD.
\end{proof}
\begin{remark}
\label{rem: FFDfield}
The converse of the Lemma does not hold. For example, let $A$ be a
commutative domain and a $k$-algebra that is not an FFD.  Moreover,
let $B$ be its field of fractions, then $B$ is vacuously an FFD since there are no nonzero non-unit elements in $B$.
\end{remark}

\begin{proof}[Proof of Theorem \ref{thm: application}] From remarks in \S 2 we see that all of the algebras (except for $G$-algebras and polynomial extensions) described have finite-dimensional filtrations whose associated graded algebras are domains.  In \S2, we also observed that all of the rings given in the list remain domains under extension of scalars and hence by Corollary \ref{thm: main2} we see that all of these algebras have the property that nonzero elements have only finitely many different factorizations up to multiplication of factors by scalars.  Hence we see that subalgebras of these algebras are also FFDs by Lemma \ref{subring}.  For $G$-algebras, we note that Theorem 2.3. of \citep{GomezRefiltering} shows that there exists a Noetherian ordering on $\NN$, satisfying the condition of the Definition \ref{def: G-algebra} and a finite-dimensional $\NN$-filtration $\mathcal{L}$ on $A$ with $\mathcal{L}_0=k$, such that ${\rm gr}_{\mathcal{L}} (A) \cong \mathcal{O}_{\bf q}(k^n)$ for some units $q_{ij}$.  Since $\mathcal{O}_{\bf q}(k^n)\otimes_k \bar{k}\cong \mathcal{O}_{\bf q}(\bar{k}^n)$ is a domain, we see that a $G$-algebra is an FFD by Corollary \ref{thm: main2}.  Finally, for polynomial rings, over these algebras, it is straightforward to see that by choosing an appropriate filtration we will again obtain an associated graded domain that behaves well with respect to extension of scalars.
\end{proof}

The algebras corresponding to common linear partial operators deserve special attention due to their ubiquity in applications.  The following result follows directly from Theorem \ref{thm: application}.

\begin{corollary} 
\label{cor: OpAlg}
For the following linear partial operators $\mathfrak{o}$,
the corresponding algebras of operators over $k[x]$, as well as finitely many tensor products of them over $k$ are finite factorization domains. Below, 
 $q\in k^*$ is fixed and $f(x)\in k[x]$.
\begin{enumerate}
\item $\mathfrak{o}(f) = \tfrac{\d f(x)}{\d x}$, the differential operator, gives rise to the first Weyl algebra $k[x]\langle \d \colon \d x = x \d + 1\rangle$;
\item $\mathfrak{o}(f) = f(x+1)$, the shift operator, gives rise to the first shift algebra $k[x]\langle s \colon s x = x s + s\rangle$;
\item $\mathfrak{o}(f) = f(q \cdot x)$, the $q$-shift or $q$-dilation operator, gives rise to the first $q$-shift algebra $k[x]\langle s_q \colon s_q x = q x s_q \rangle$;
\item $\mathfrak{o}(f) = \tfrac{f(qx) - f(x)}{(q-1)x}$, the $q$-differential operator, gives rise to the first $q$-Weyl algebra $k[x]\langle \d_q \colon \d_q x = q x \d_q + 1\rangle$;
\item $\mathfrak{o}(f) = \int_0^x f(t) dt$, the integration operator, gives rise to the graded algebra $k[x]\langle I : I x = x I -I^2 \rangle$.
\end{enumerate}
\end{corollary}

Although the commutative domains has been extensively studied \citep{anderson1990factorization, anderson1992elasticity,anderson1996finite, anderson1997factorization}, we give a concise proof of the FFD property for polynomial rings using our methods.

\begin{proposition} 
Let $k$ be any field. Then $k[z_i \colon i\in\NN]$ and $k[z_i \colon i\in\ZZ]$ are FFDs.
\end{proposition}
\begin{proof}
Let us view $k[z_1, z_2, \ldots]$ as $\NN$-graded algebra with $\deg(z_i)=i$. Then each graded piece is finite-dimensional over $k$ and it is straightforward to see that the algebras behave well under extension of scalars.  Hence by the Corollary \ref{thm: main2}, the algebra is an FFD.  Since $k[z_i \colon i\in\ZZ]$ and $k[z_i\colon i\in \NN]$ are isomorphic as $k$-algebras, 
$k[z_i \colon i\in\ZZ]$ is an FFD as well.
\end{proof}

\begin{corollary}
  Let $k$ be any field. Then $k[z_i^{\pm 1} \colon i\in\NN]$ and
  $k[z_i^{\pm 1} \colon i\in\ZZ]$ are FFDs.
\end{corollary}
\begin{proof}
Any non-unit element of $k[z^{\pm 1}_{i} \colon i\in\NN]$ (resp. $k[z^{\pm 1}_{i} \colon i\in\ZZ]$) can be written uniquely as
\[
\left( \prod_{{\rm finite}} z_i^{-k_i}\right) \cdot p(z), \quad k_i\in\NN, \quad p(z) \in k[z_{i }\colon i\in\NN].
\]
Thus, since $k[z_i \colon i\in\NN]$ (resp $k[z_{i} \colon
  i\in\ZZ]$) is an FFD and $z_{ i}^{-1}$ are units, we obtain the
desired result. 
\end{proof}

\section{Examples of non-FFDs and Remarks}
In this section we give some examples of non-FFDs, which will be useful in giving a fuller understanding of the FFD property.  We begin with a proposition which gives examples of $\mathbb{N}$-graded domains (over any field) that are not FFDs.  We note, however, that these rings all have some infinite-dimensional homogeneous components and hence these examples do not contradict Theorem \ref{thm: main} in the case when the base field is algebraically closed.

\begin{proposition}
Let $k$ be any field. Consider $R=k[u][\{x_n, y_n\colon n\in\ZZ\}]$, a polynomial ring in infinitely many commuting variables. Let $\sigma: R\to R$ be the automorphism $x_n \mapsto x_{n+1}, y_n\mapsto y_{n+1}, u\mapsto u$. Moreover, let  
$A=R[t,t^{-1};\sigma]$ be the Laurent Ore extension of $A$; that is, for every
$r \in R$ we have the commutation rules $t^{\pm 1} \cdot r = \sigma^{\pm 1}(r)\cdot t^{\pm 1}$ and $t t^{-1}=t^{-1}t = 1$.  
Then the following $k$-algebras are not FFDs. 
\begin{itemize}
\item[a)] $R/ \langle \{ x_i y_i -u \colon i\in \ZZ \} \rangle$, which is an infinitely generated $k$-algebra.
\item[b)] $A/{}_{A} \langle x_0 y_0 - u \rangle_A$, which is a finitely generated $k$-algebra.
\end{itemize}
\end{proposition}

\begin{proof}
{\bf a)}. Let $G:=\{ x_i y_i -u \colon i\in \ZZ \} \subset R$ and $J:=\langle G \rangle \subset R$. Then $G$ is a Gr\"obner basis of $J$ with respect to any monomial well-ordering $\prec$ on $R$, such that $u\prec x_i, y_j, x_k y_k$. In particular, $\prec$ can be taken to be an elimination ordering for $x_i, y_i$.
Then the leading monomial of a generator from $G$ is $x_i y_i$ and hence, by the product criterion \citep{buchberger1979criterion}, $G$ is a a Gr\"obner basis of $J$ with respect to $\prec$. Now, suppose that there exists some $p\in k[u]\setminus\{0\}$, such that 
 $p(u) \cdot (r+J)=0 \in R/J$. We can assume that a representative $r$ is already completely reduced with respect to $J$, that is no monomial of $r\neq 0$ is divisible by any $x_j y_j$. Then $p(u) r \in J \Rightarrow$ $\lm(p) \lm(r) \in \lm(J)=\langle { x_i y_i \colon i \in \ZZ } \rangle$, where $\lm(f)$ simply denotes the leading monomial of $f$.  Since 
 $\lm(p)\in k[u]\setminus\{0\}$, it follows that $\lm(r) \in \lm(J)$, which is a contradiction.
  Thus we have demonstrated, that the $R$-module $R/J$ has no 
  $k[u]$-torsion.

This means, that the homomorphism of $k$-algebras from $R/J$ to the localization $S^{-1}(R/J)$ of $R/J$ with respect to $S=k[u]\setminus\{0\}$  is injective, that is we have an embedding
\[
k[u][\{x_n, y_n\colon n\in\ZZ\}]/\langle \{ x_i y_i -u \colon i\in \ZZ \} \rangle\]
into
\[ 
k(u)[\{x_n, y_n\colon n\in\ZZ\}]/\langle \{ x_i y_i -u \colon i\in \ZZ \} \rangle.
\]
Moreover, the latter algebra is an integral domain, since the homomorphism of $k(u)$-algebras 
\[
\varphi: k(u)[\{x_n, y_n\colon n\in\ZZ\}]/\langle \{ x_i y_i -u \colon i\in \ZZ \} \rangle
\to k(u)[\{z_n, z^{-1}_n\colon n\in\ZZ\}], 
\]
$\varphi(x_n)=z_n, \varphi(y_n) = u z^{-1}_n$, is an isomorphism.\\
Indeed, the surjectivity is clear since $$k(u)[\{ z_n, u z^{-1}_n \colon n\in\ZZ\}]=k(u)[\{z_n, z^{-1}_n\colon n\in\ZZ\}].$$ For injectivity, we use Gr\"obner bases for the computation of $\ker \varphi  $. Let $F := \{ z_i - x_i, uw_i - y_i, z_i w_i - 1 \colon i\in\ZZ\} \subset k(u)[\{x_n, y_n, z_n, w_n\colon n\in\ZZ\}]$ and $G_F$ be a 
Gr\"obner basis of $F$ with respect to a monomial ordering, eliminating $z_i, w_i$ (this means any monomial, not containing $z_i, w_i$ is smaller than any monomial, containing at least one of them).
Then, by e.g. Lemma 1.8.16 of \citep{GPS}, one has that $G_F\cap k(u) [\{x_n, y_n\colon n\in\ZZ\}]$ is a Gr\"obner basis of $\ker \varphi$.

Denote the polynomials in $F$ by $f_1, f_2, f_3$  for a fixed $i$. Then $f_3 \to_{f_2} uf_3 - z f_2 = y_i z_i -u \to_{f_1} x_i y_i -u =: f'_3$. Now
the leading monomials are $\{ z_i, w_i, x_i y_i \}$ and thus are pairwise coprime.
Moreover, for $j\neq i$ the corresponding monomials with index $j$ and
those with index $i$ are coprime. Hence, by the product criterion \citep{buchberger1979criterion}, 
$G_F =  \{ z_i - x_i, uw_i - y_i, x_i y_i -u \colon i\in\ZZ\}$ is a Gr\"obner basis and therefore the kernel is generated via a Gr\"obner basis $F'\cap k(u) [\{x_n, y_n\colon n\in\ZZ\}] = \{  x_i y_i -u \colon i\in\ZZ\}$, which is precisely the factor ideal
from the statement.  Thus, $R/J$ is an integral domain which is not an FFD, since the non-invertible polynomial $u$ has infinitely many factorizations $u = x_i y_i$ with $i\in\ZZ$. Note that $R/J$ can be viewed as $\NN$-graded algebra by assigning $\deg(u)=2, \deg(x_i)=\deg(y_i)=1$ and the arguments above carry over. Hence, with respect to this grading, $R/J$ is an example
of an $\NN$-graded domain, which is not an FFD. \\

{\bf b)} $A$ is canonically generated by $t, t^{-1}, u$ and $\{x_n, y_n\colon n\in\ZZ\}$. On the other hand,  for all $r\in R$ from the relation 
$t^{\pm 1} \cdot r = \sigma^{\pm 1}(r)\cdot t^{\pm 1}$ we conclude, that
$\sigma^{\pm 1}(r) = t^{\pm 1} r t^{\mp 1}$. In particular,
\[
x_{n+1} = t x_n t^{-1}, \ y_{n-1} = t^{-1} y_n t
\]
holds and hence $A$ is finitely generated as $k$-algebra by $t, t^{-1}, u , x_0, y_0$ subject to relations $\{ [u,x_0]=[u,y_0]=[x_0, y_0] = [t,u] = [t^{-1}, u] = [t^{-1}, t] = t^{-1}t - 1 =0 \}$, so it is finitely presented. Notably, in this presentation $t^{\pm 1}$ are free over $k[x_0,y_0]$ 
since the relation $t^{\pm 1} p(x_0,y_0) = \sigma^{\pm 1} ( p(x_0,y_0) ) t^{\pm 1}$, by expanding $\sigma(p) = t p t^{-1}$ on the right, becomes tautologic $t^{\pm 1} p(x_0,y_0) = t^{\pm  1} p(x_0,y_0)$. We will refer to
this presentation as $\star$ in what follows.\\
Now let $I$ be the two-sided ideal of $A$, generated by $\{x_0 y_0 - u\}$.
Then, on one hand, the $k$-algebra $A/I$ is finitely generated and finitely presented, since in the presentation $\star$ we have added just one relation. 
On the other hand, for all $i\in \ZZ$, $I \ni t^i (x_0 y_0 - u) t^{-i} = x_i y_i -u \in G$, where
$G$ was defined in (a) above. Since $G \subset R$ is $\sigma$-invariant and a Gr\"obner basis, the ideal $J = \langle G \rangle \subset R$ is $\sigma$-invariant. The same arguments show also the $\sigma^{-1}$-invariance. 
Thus $J$ is an $A$-submodule of $R$ and indeed
$\sigma$ induces an automorphism of an integral domain $R/J$. Then 
$(R/J)[t,t^{-1}; \sigma]$ is an integral domain as well. As we observe, 
$A/{}_{A} \langle x_0 y_0 - u \rangle_A = (R/J)[t,t^{-1}; \sigma]$ and the claim follows, since $u = x_i y_i, i\in\ZZ$ is an infinite family of factorizations.
Note, that $R/J$ can be viewed as $\NN$-graded domain. However, 
one can make $(R/J)[t,t^{-1}; \sigma]$ nontrivially graded only by considering
$\deg(t) = - \deg(t^{-1})$, thus only a $\ZZ$-grading exists.
\end{proof}


\begin{example}
\label{ex:EasyNonFFD}
{\em 
Let us consider yet another example. Let $\sigma: \mathbb{C} \to \mathbb{C}$ be the complex conjugation automorphism and let $R = \mathbb{C}[t; \sigma]$. Then $R$ is a domain that it is not an FFD.  To see this, we observe that $t^2- 1$ has infinitely many factorizations 
$(t - e^{-i\phi})(t + e^{i\phi})$, where $\phi\in [0,2\pi)$.  The center of $R$ is easily checked to be the ring $\mathbb{R}[t^2]$ and so the only central units are elements of $\mathbb{R}^*$.  Thus we see that the factorizations given above are distinct up to multiplication by central units and so $R$ is not an FFD.

We give some additional observations:
\begin{enumerate}
\item[(i)] $R \cong \mathbb{R} \langle i, t \rangle/ \langle t i + i t, i^2+1 \rangle$ is a
finitely presented $\mathbb{R}$-algebra: the relations $i^2=-1$ and $i t=- t i$ allow us to write any element of $R$ in the form $p_0(t)+p_1(t)i$, with
$p_0(t), p_1(t)\in \mathbb{R}[t]$, and these are linearly independent
over $\mathbb{R}$ by definition of the skew polynomial ring. \\
\item[(ii)]  Filtration/grading: from the relation $i^2 +1 = 0$ it follows, that $\deg(i)=0$ is the only possibility. Then $R$ is an $\NN_0$-graded domain with $\deg(t)$ being 1 and $R_n := \sum_{j=0}^n \mathbb{C} \cdot t^n$ is finite-dimensional both over $\mathbb{R}$.\\
\item[(iii)] $\mathbb{C} \langle y, t \rangle/ \langle t y + y t, y^2+1 \rangle$ is not a domain since $0 = y^2+1= (y-i)(y+i)$ holds. \\
\item[(iv)] There is an interpretation of (iii) in terms of tensor products: since both $R$ and $\mathbb{C}$ are $\mathbb{R}$-algebras, we may form the tensor product $D:=R \otimes_{\mathbb{R}} \mathbb{C} = R \otimes_{\mathbb{R}} \bar{\mathbb{R}}$. Taking $R$ in variables $y,t$ as above, we see that in $D$ we have $0 = (y^2+1)\otimes 1 = y^2\otimes 1 + 1\otimes 1 = (y\otimes 1 - 1 \otimes i)\cdot (y\otimes 1 + 1 \otimes i)$, and hence shows $D$ is not a domain. \\
\item[(v)] since $\sigma\mid_{\mathbb{R}} = id_{\mathbb{R}}$, there is a
subalgebra $T=\mathbb{R}[t;\sigma\mid_{\mathbb{R}}] = \mathbb{R}[t] \subset R$, which is an FFD.
\end{enumerate}}
\end{example}


\begin{remark}
\label{rem: RatWeyl}
In general, the property of not being an FFD does not pass to localizations, not even in
the commutative case (Remark \ref{rem: FFDfield}). In the noncommutative case, the property of being an FFD does not in general pass to localizations.  As an example, consider the first polynomial Weyl algebra $A_1$ over a field $k$, which is an FFD. However, its Ore localization with respect to $k[x]\setminus\{0\}$, denoted by $B_1$, is not an FFD.  It is straightforward to check that the central units of $B_1$ are precisely the elements of $k^*$ when $k$ has characteristic zero and are $k^* x^{p\mathbb{Z}}$ if $k$ has characteristic $p>0$.  Also, $B_1$ is generated by $\d$ over the transcendental field 
extension $\mathbb{C}(x)$ subject to the relation $\d g(x) = g(x) \d + \tfrac{\d g(x)}{\d x}$ for $g\in \mathbb{C}(x)$. A classical example shows, that 
\[
\forall (b,c)\in \mathbb{C}^2\setminus\{(0,0)\}, \quad 
\d^2 = \left( \d + \frac{b}{bx-c} \right) \cdot \left(\d - \frac{b}{bx-c} \right),
\]
thus there are infinitely many distinct factorizations in $B_1$ up to multiplication by central units. In
fact, one has a complete description of the form of all possible
factorizations of $\d^2$ in the equation above. To prove this, one considers a
factorization $\d^2 = \phi \cdot \psi$, where $\phi,\psi$ are in $A_1$
and not units. Then $\phi$ and $\psi$ are of
degree one in $\d$, and without loss of generality one can assume that
they are normalized, i.e. they have the form $\phi = \d + f$ and $\psi
= \d + g$ for $f,g \in k(x)$. Using a coefficient comparison of the
product $\phi \cdot \psi$, one can derive that $f = -g$ and that $f$
has to be a rational function solution of the ordinary differential
equation $\frac{\d f}{\d x} = -f^2$. Besides the trivial solution $f
\equiv 0$, the only possible other solution is given by $f(x)=
\frac{b}{bx-c}$ for some constants $(b,c) \in k^2\setminus\{(0,0)\}$.
\end{remark}

Via somewhat analogous examples, one can show that both rational shift
and $q$-shift algebras are not finite factorization domains either.

\begin{remark}
\label{rem: RatShift}
Also over the first rational shift algebra one encounters a phenomenon as before. This example was communicated to us by Michael Singer. 
Namely, let $(c_1, c_2) \in k^2\setminus\{(0,0)\}$. Then
\[
s^2 - 2(n+2) s + (n+2)(n+1) = 
\left(s - (n+2)\frac{c_1 n + c_2}{c_1 (n+1) + c_2}\right) \cdot
\left(s - (n+1)\frac{c_1 (n+1) + c_2}{c_1 n + c_2}\right).
\]
In particular, for $c_1=1$ and $c_2=c$ we obtain an example, somewhat similar to Example \ref{rem: RatWeyl}, namely
\[
s^2 - 2(n+2) s + (n+2)(n+1) = 
\left(s - (n+2)(1-\frac{1}{n+ c +1 }) \right)\cdot \left(s - (n+1)\left(1 + \frac{1}{n + c}\right) \right).
\]
\end{remark}

\begin{remark}
\label{rem: RatQShift}
The construction of an element with infinitely many factorizations in
the first rational $q$-shift algebra is similar to the one mentioned in
Remark \ref{rem: RatShift}.

Namely, let $(c_1, c_2) \in k^2\setminus\{(0,0)\}$. Then
\[
s_q^2 - (1+q) s_q + q = 
\left(s_q - \left(1 - \frac{c_2(1-q)}{c_1 x + c_2(n+1)}\right) \right) \cdot
\left(s_q - \left(q + \frac{c_2(1-q)}{c_1 x + c_2(n+1)}\right) \right).
\]
In particular, for $c_1=c-1$ and $c_2=1$ we obtain a similar easy example like in Remarks \ref{rem: RatWeyl} and \ref{rem: RatShift}, namely
\[
s_q^2 - (1+q) s_q + q = 
\left(s_q - \left(1 - \frac{1-q}{c x +1}\right) \right) \cdot
\left(s_q - \left(q + \frac{1-q}{c x +1}\right) \right).
\]
\end{remark}

\begin{remark}
  When applying the notion of similarity as used by Ore in
  \citep{ore1933} to
  characterize different possible divisors of a given
  element, the infinitely many factorizations presented in Remark
  \ref{rem: RatWeyl}, \ref{rem: RatShift} and \ref{rem: RatQShift} 
  are
  the same up to similarity. Moreover, it holds that the rational Weyl,
  the rational shift and the rational $q$-shift 
  algebra are
   unique
  factorization domains up to this notion. The examples
   illustrate
  that the theoretically interesting property of similarity ignores the sizes of coefficients of different
  normalized factors, which is inconvenient from a computational perspective.
\end{remark}

\section{Conclusion}

A proof of the finite factorization property of not necessarily
commutative $k$-algebras with filtration, where
the associated graded ring is a domain, has been presented in the case where the base field is either finite or algebraically closed. Moreover,
an upper bound on the number of possible factorizations of a given element has
been established. The proof applies to well-studied algebras such as the
Weyl algebras, the shift-algebras, coordinate rings of quantum affine spaces and
enveloping algebras of finite dimensional Lie algebras, and for these examples one does not require that the base field be algebraically closed.

As a consequence of the presented results, one can now begin to
develop systematic algorithmic ways to treat the problem of finding all factorizations of an element in these $k$-algebras. 
Indeed, some approaches to computation of factorizations were explored in the past. We have presented for the first time the crucial argument for the termination of a factorization procedure (e.~g. the one, based on ansatz) over these algebras, thus proving that a factorization algorithm exists and terminates. This opens a perspective towards algorithmic solution and numerous applications of factorizations. 
Classically, dealing with systems of linear partial functional equations with
variable coefficients, any factorization of an operator (a polynomial, representing an equation as above) provides some insight on its solutions via
an analogue of the partial cascade integration for differential equations. 
Notably, there are several ``algorithms" for computing factorizations of
elements of algebras of operators over, say, the field of rational functions. By
our result we claim, that these are merely \emph{procedures}, since in general these algebras are not FFDs, hence the output contains infinite data and thus the procedure will not terminate.

Suppose that in a $k$-algebra $R$, $p\in R$ possesses finitely many factorizations, say $p=f_i g_i, 1\leq i \leq m$.  Then for any left ideal $L \subset R$ one has $L + Rp \subseteq \cap_i (L + Rg_i)$. Especially in the cases where the inequality becomes equality one has some analogue of a primary decomposition. The possibility to construct ideals, containing the
given one, is useful for a number of situations.

Factorization over a Noetherian domain $R$ can be used as a preprocessing 
to the factorization over an Ore localization $S^{-1}R$, where $S \subset R$ is
a multiplicatively Ore set. It is especially useful, when $R$ is an FFD, since 
$S^{-1}R$ will often lose this property as indicated in the paper. Moreover, 
practical computations over $S^{-1}R$ are generally more complicated than
those over $R$. By following the fraction-free strategy one can perform many
computations over $S^{-1}R$ by invoking only manipulations over $R$. In this respect, knowing that any $p\in R$ has finitely many factorizations and being able to compute them, one could
solve the localized problem: does there exist $s\in S, \tilde{p}\in R$, such that $p = s \tilde{p}$? If $p$ is from a left ideal $L$ and $p=s_1 \tilde{p}_1 = \ldots = s_m \tilde{p}_m$ are all factorizations of this shape, then $\tilde{p}_1, \ldots, \tilde{p}_m \in S^{-1} L$. This is also connected to the $S$-torsion of $R$-modules.

In \citep{GrigorievSchwartz:2007} the authors presented procedures for computing a primary decomposition and a Loewy-decomposition (a dual concept to the previous) of a module over the $n$-th Weyl algebra over an universal differential field. Having the algorithm for finite factorization for the $n$-th Weyl algebra with polynomial coefficients, one could give a generalization
of these decomposition techniques.  

A recent paper \citep{GHL14} presents a new powerful algorithm for
factorization of polynomials over the $n$-th Weyl algebra as well as
its implementation \texttt{ncfactor.lib} in the computer algebra system \textsc{Singular}.

As part of future work, we see a great deal of potential that the bounds on the number of distinct factorizations can be optimized when considering specific $k$-algebras in isolation.
Yet another problem of practical importance concerns the question on an
element $p$ from a domain, which is non-FFD: is there an algorithm, deciding whether $p$ has finitely many factorizations? Is there an upper bound for
the number of factorizations in the finite case? 

Another subject of future research is to generalize the
terminology and the results in \citep{anderson1990factorization, anderson1997factorization} to noncommutative domains. If extendable, the more refined categorization of integral domains with
respect to factorization properties of their elements might contribute towards a better understanding
of certain noncommutative algebras.

\section*{Acknowledgment}
We thank the referee for many helpful suggestions and remarks.
We are grateful to Michael Singer, Mark Giesbrecht, Dima
Grigoriev and Vladimir Bavula for the fruitful discussions we had with
them.

 \end{document}